\documentclass[11pt]{amsart}

\theoremstyle{plain}
\newtheorem{theorem}                 {Theorem}      [section]
\newtheorem{proposition}  [theorem]  {Proposition}
\newtheorem{corollary}    [theorem]  {Corollary}
\newtheorem{conjecture}   [theorem]  {Conjecture}

\newtheorem*{theorem*}               {Theorem\;{\bf \ref{th:
bih surf s4_parallel H}}}

\theoremstyle{definition}

\newtheorem{remark}       [theorem]  {Remark}
\newtheorem{defi}   [theorem]  {Definition}

\numberwithin{equation}{section}

\def \1{\mbox{${\mathbf 1}$}}

\DeclareMathOperator{\trace}{trace}
\DeclareMathOperator{\grad}{grad}

 \DeclareMathOperator{\Id}{Id}

\makeindex

\begin{document}

\title[Properties of biharmonic submanifolds in spheres]
{PROPERTIES OF BIHARMONIC SUBMANIFOLDS IN SPHERES\footnotemark}
\author{A. Balmu\c s}
\address{Faculty of Mathematics, ``Al.I.~Cuza'' University of Iasi\\
\newline
Bd. Carol I Nr. 11 \\
700506 Iasi, ROMANIA} \email{adina.balmus@uaic.ro,
oniciucc@uaic.ro}
\author{S. Montaldo}
\address{Universit\`a degli Studi di Cagliari\\
Dipartimento di Matematica\\
\newline
Via Ospedale 72\\
09124 Cagliari, ITALIA} \email{montaldo@unica.it}
\author{C. Oniciuc}

\begin{abstract}
In the present paper we survey the most recent classification
results for proper biharmonic submanifolds in unit Euclidean
spheres. We also obtain some new results concerning geometric
properties of proper biharmonic constant mean curvature submanifolds
in spheres.
\end{abstract}

\maketitle

\stepcounter{footnote}\footnotetext{The second author was supported
by the grant Start-up, University of Cagliari. The third author was
supported by PCE Grant PNII\textunderscore 2228 (502/2009),
Romania.}

\section{Introduction}\label{sec:1}

%%Ever since 1964, when J. Eells and J.H. Sampson published their
%%first survey on harmonic maps (see \cite{ES}), there have been
%%many attempts to generalize the notions of harmonic map and
%%harmonic (minimal) Riemannian immersion. One can, for example,
%%consider the variational problems associated to the functionals
%%obtained by integrating the squared norm of the tension field or
%%of the mean curvature vector field. Thus, {\it biharmonic maps}
%%are the critical points of the {\it bienergy functional}
%%$$
%%E_2:C^{\infty}(M,N)\to\mathbb{R} , \quad
%%E_2(\varphi)=\frac{1}{2}\int_{M}\, |\tau(\varphi)|^2\,v_g
%%$$
%%and {\it Willmore immersions} are the critical points of the {\it
%%Willmore functional}
%%$$
%%W:\mathrm{Imm}(M^2,N)\to\mathbb{R} , \quad W(\varphi)=\int_{M^2}\,
%%\big(|H|^2+K\big) \,v_{\varphi^{\ast}h}
%%$$
%%where $K$ is the sectional curvature of $(N,h)$ restricted to the
%%image of $M^2$. Although these variational problems lead to
%%natural generalizations of harmonic maps and minimal immersions,
%%biharmonic Riemannian immersions do not recover Willmore
%%immersions, not even when the ambient space is $\mathbb{R}^n$.

{\it Biharmonic maps} $\phi:(M,g)\to(N,h)$ between Riemannian
manifolds are critical points of the {\it bienergy functional}
$$
E_2(\phi)=\frac{1}{2}\int_{M}\,|\tau(\phi)|^2\,v_g,
$$
where $\tau(\phi)=\trace\nabla d\phi$ is the tension field of $\phi$
that vanishes for harmonic maps (see \cite{ES}). The Euler-Lagrange
equation corresponding to $E_2$ is given by the vanishing of the
{\it bitension field}
$$
\tau_2(\phi)=-J^{\phi}(\tau(\phi))=-\Delta\tau(\phi) -\trace
R^N(d\phi,\tau(\phi))d\phi,
$$
where $J^{\phi}$ is formally the Jacobi operator of $\phi$ (see
\cite{J1}). The operator $J^{\phi}$ is linear, thus any harmonic map
is biharmonic. We call {\it proper biharmonic} the non-harmonic
biharmonic maps. Geometric and analytic properties of proper
biharmonic maps were studied, for example, in \cite{PBAFSO, TL,RM}.

The submanifolds with non-harmonic (non-minimal) biharmonic
inclusion map are called {\it proper biharmonic submanifolds}.
Initially encouraged by the non-existence results for proper
biharmonic submanifolds in non-positively curved space forms (see,
for example, \cite{CMO2, BYC_ISH,D,HV}), the study of proper
biharmonic submanifolds in spheres constitutes an important research
direction in the theory of proper biharmonic submanifolds.

The present paper is organized as follows.

Section \ref{sec:2} is devoted to the main examples of proper
biharmonic submanifolds in spheres and to their geometric
properties, mainly regarding the type and the order in the sense of
Chen. Also, it gathers the most recent classification results for
such submanifolds (for detailed proofs see \cite{B}).

In Section \ref{sec:3} we present a series of new results concerning
geometric properties of proper biharmonic constant mean curvature
submanifolds in spheres. We begin with some identities which hold
for proper biharmonic submanifolds with parallel mean curvature
vector field (Propositions~\ref{prop: prop||MC} and \ref{prop:
prop||MC_from_Gauss_eq}). We then obtain some necessary conditions
that must be fulfilled by proper biharmonic constant mean curvature
submanifolds (Corollary~\ref{cor: conseq_type}) and we end this
section with a refinement, for hypersurfaces, of a result on the
estimate of the mean curvature of proper biharmonic submanifolds in
spheres (Theorem~\ref{th: raf_tip_hyper}).

The fourth section presents two open problems concerning the
classification of proper biharmonic hypersurfaces and the mean
curvature of proper biharmonic submanifolds in spheres.

In the last section we briefly present an interesting link between
proper biharmonic hypersurfaces and $II$-minimal hypersurfaces in
spheres.

Other results on proper biharmonic submanifolds in spaces of
non-constant sectional curvature can be found, for example, in
\cite{CMP, YJCHS, FO, FLMO, I, OU2, OU1}.

\section{Proper biharmonic submanifolds in spheres}\label{sec:2}

The attempt to obtain classification results for proper biharmonic
submanifolds in spheres was initiated with the following
characterization theorem.

\begin{theorem}[\cite{O2}]\label{th: bih subm S^n}
{\rm (i)} The canonical inclusion $\phi:M^m\to\mathbb{S}^n$ of a
submanifold $M$ in an $n$-dimensional unit Euclidean sphere is
biharmonic if and only if
\begin{equation}\label{caract_bih_spheres}
\left\{
\begin{array}{l}
\ \Delta^\perp H+\trace B(\cdot,A_H(\cdot))-mH=0
\\ \mbox{} \\
\ 4\trace A_{\nabla^\perp_{(\cdot)}H}(\cdot)+m\grad(\vert H
\vert^2)=0,
\end{array}
\right.
\end{equation}
where $A$ denotes the Weingarten operator, $B$ the second
fundamental form, $H$ the mean curvature vector field,
$\nabla^\perp$ and $\Delta^\perp$ the connection and the Laplacian
in the normal bundle of $M$ in $\mathbb{S}^n$, and $\grad$ denotes
the gradient on $M$.

If $M$ is a submanifold with parallel mean curvature vector field in
$\mathbb{S}^n$, then $M$ is biharmonic if and only if $\trace
B(\cdot,A_H(\cdot))=mH$.

{\rm (ii)} A hypersurface $M$ with nowhere zero mean curvature
vector field in $\mathbb{S}^{m+1}$ is biharmonic if and only if
\begin{equation}\label{eq: caract_bih_hipersurf_spheres}
\left\{
\begin{array}{l}
\Delta^\perp H-(m-|A|^2)H=0
\\ \mbox{} \\
\ 2A(\grad (|H|))+m|H|\grad(|H|)=0.
\end{array}
\right.
\end{equation}
If $M$ is a non-zero constant mean curvature hypersurface in
$\mathbb{S}^{m+1}$, then $M$ is proper biharmonic if and only if
$|A|^2=m$.
\end{theorem}

We note that the compact minimal, i.e. $H=0$, hypersurfaces with
$|A|^2=m$ in $\mathbb{S}^{m+1}$ are the Clifford tori
$\mathbb{S}^k(\sqrt{k/m})\times\mathbb{S}^{m-k}(\sqrt{(m-k)/m})$,
$1\leq k\leq m-1$ (see \cite{CCK}).

Before presenting the basic examples of proper biharmonic
hypersurfaces in spheres, together with some of their geometric
properties, we recall the following definition (see, for example,
\cite{BYC1}), which shall be used throughout the paper.

\begin{defi}
An isometric immersion of a compact manifold $M$ in
$\mathbb{R}^{n}$, $\varphi:M\to \mathbb{R}^{n}$, is called of {\it
k-type} if its spectral decomposition contains exactly $k$ non-zero
terms, excepting the center of mass $\varphi_0=\frac{1}{\mathrm{Vol}
(M)}\int_M\varphi\,v_g$. More precisely,
$$
\varphi=\varphi_0+\sum_{t=p}^q \varphi_t,
$$
where $\Delta\varphi_t=\lambda_t\varphi_t$ and
$0<\lambda_1<\lambda_2<\cdots\uparrow\infty$.

The pair $[p,q]$ is called {\it the order of the immersion}
$\varphi:M\to \mathbb{R}^{n}$.
\end{defi}

\newpage

\subsection{The main examples of proper biharmonic submanifolds in spheres}\label{sec:2.1}

\subsubsection*{The hypersphere $\mathbb{S}^{m}(1/{\sqrt
2})\subset\mathbb{S}^{m+1}$} \quad

\noindent Consider $\mathbb{S}^{m}(a)=\Big\{(x^1,\ldots, x^{m},
x^{m+1},b)\in\mathbb{R}^{m+2}: |x|=a\Big\}\subset\mathbb{S}^{m+1}$,
where $a^2+b^2=1$. If $H$ is the mean curvature vector field of
$\mathbb{S}^{m}(a)$ in $\mathbb{S}^{m+1}$, one gets $\nabla^\perp
H=0$, $|H|=\frac{|b|}{a}$ and $|A|^2=m\frac{b^2}{a^2}$.

Theorem \ref{th: bih subm S^n} implies that $\mathbb{S}^{m}(a)$ is
proper biharmonic in $\mathbb{S}^{m+1}$ if and only if $a=1/{\sqrt
2}$ (see \cite{CMO1}).

\subsubsection*{The generalized Clifford torus $\mathbb{S}^{m_1}(1/{\sqrt 2})
\times\mathbb{S}^{m_2}(1/{\sqrt 2})\subset\mathbb{S}^{m+1}$}\quad

\noindent The generalized Clifford torus,
$M=\mathbb{S}^{m_1}(1/{\sqrt 2})\times\mathbb{S}^{m_2}(1/{\sqrt
2})$, $m_1+m_2=m$, $m_1\neq m_2$, was the first example of proper
biharmonic submanifold in $\mathbb{S}^{m+1}$ (see \cite{J1}).

Consider
\begin{eqnarray*}
M&=&\Big\{(x^1,\ldots,x^{m_1+1}, y^1,\ldots,
y^{m_2+1})\in\mathbb{R}^{m+2}: |x|=a_1,
|y|=a_2\Big\}\\
&=&\mathbb{S}^{m_1}(a_1)\times\mathbb{S}^{m_2}(a_2)\subset\mathbb{S}^{m+1},
\end{eqnarray*}
where $a_1^2+a_2^2=1$. Then $\nabla^\perp H=0$,
$|H|=\frac{1}{a_1a_2m}|a_2^2m_1-a_1^2m_2|$ and
$|A|^2=\left(\frac{a_2}{a_1}\right)^2m_1+\left(\frac{a_1}{a_2}\right)^2m_2$.

From Theorem \ref{th: bih subm S^n} it follows that $M$ is proper
biharmonic in $\mathbb{S}^{m+1}$ if and only if $a_1=a_2=1/{\sqrt
2}$ and $m_1\neq m_2$ (see, also, \cite{CMO2}).

Inspired by these basic examples, two methods for constructing
proper biharmonic submanifolds of codimension higher than one in
$\mathbb{S}^n$ were given.

\begin{theorem}[\cite{CMO2}]\label{th: rm_minim}
Let $M$ be a minimal submanifold of\,
$\mathbb{S}^{n-1}(a)\subset\mathbb{S}^n$. Then $M$ is proper
biharmonic in $\mathbb{S}^{n}$ if and only if $a=1/{\sqrt 2}$.
\end{theorem}

\begin{remark}
(i) This result, called the {\it composition property}, proved to be
quite useful for the construction of proper biharmonic submanifolds
in spheres. For instance, it implies the existence of closed
orientable embedded proper biharmonic surfaces of arbitrary genus in
$\mathbb{S}^4$ (see \cite{CMO2}).

(ii) All minimal submanifolds of $\mathbb{S}^{n-1}(1/{\sqrt
2})\subset\mathbb{S}^n$ are pseudo-umbilical, i.e. $A_H=|H|^2\Id$,
with pa\-rallel mean curvature vector field in $\mathbb{S}^n$ and
$|H|=1$.

(iii) Denote by $\phi:\mathbb{S}^{m}(1/{\sqrt
2})\to\mathbb{S}^{m+1}$ the inclusion of $\mathbb{S}^{m}(1/{\sqrt
2})$ in $\mathbb{S}^{m+1}$ and by
$\mathbf{i}:\mathbb{S}^{m+1}\to\mathbb{R}^{m+2}$ the canonical
inclusion. Let $\varphi:\mathbb{S}^{m}(1/{\sqrt
2})\to\mathbb{R}^{m+2}$, $\varphi=\mathbf{i}\circ \phi$, be the
inclusion of $\mathbb{S}^{m}(1/{\sqrt 2})$ in $\mathbb{R}^{m+2}$.
Then
\begin{equation}\label{eq: spectr_decomp_min}
\varphi=\varphi_0+\varphi_p,
\end{equation}
where $\varphi_0, \varphi_p:\mathbb{S}^{m}(1/{\sqrt
2})\to\mathbb{R}^{m+2}$, $ \varphi_0(x,1/\sqrt 2)=(0,1/\sqrt 2)$,
$\varphi_p(x,1/\sqrt 2)=(x,0)$ and $ \Delta\varphi_p=2m\varphi_p$.

Thus $\mathbb{S}^{m}(1/{\sqrt 2})$ is a $1$-type submanifold of
$\mathbb{R}^{m+2}$ with center of mass in $\varphi_0=(0,1/\sqrt 2)$
and eigenvalue $\lambda_p=2m$, which is the first eigenvalue of the
Laplacian on $\mathbb{S}^{m}(1/{\sqrt 2})$, i.e. $p=1$.

Moreover, it is not difficult to verify that all minimal
submanifolds in $\mathbb{S}^{m}(1/{\sqrt
2})\subset\mathbb{S}^{m+1}$, as submanifolds in $\mathbb{R}^{m+2}$,
have the spectral decomposition given by \eqref{eq:
spectr_decomp_min}.
\end{remark}

Non pseudo-umbilical examples were also produced by proving the
following {\it product composition property}.

\begin{theorem}[\cite{CMO2}]\label{th:hipertor}
Let $M_1^{m_1}$ and $M_2^{m_2}$ be two minimal submanifolds of
$\mathbb{S}^{n_1}(a_1)$ and $\mathbb{S}^{n_2}(a_2)$, respectively,
where $n_1+n_2=n-1$, $a_1^2+a_2^2=1$. Then $M_1\times M_2$ is proper
biharmonic in $\mathbb{S}^n$ if and only if $a_1=a_2=1/{\sqrt 2}$
and $m_1\neq m_2$.
\end{theorem}

\begin{remark}
(i) The proper biharmonic submanifolds of $\mathbb{S}^n$ constructed
as above are not pseudo-umbilical, but they still have parallel mean
curvature vector field, thus constant mean curvature, and
$|H|=\frac{|m_2-m_1|}{m_1+m_2}\in (0,1)$.

(ii) Let $\varphi:\mathbb{S}^{m_1}(1/{\sqrt 2})
\times\mathbb{S}^{m_2}(1/{\sqrt 2})\to\mathbb{R}^{m+2}$ be the
inclusion of $\mathbb{S}^{m_1}(1/{\sqrt 2})
\times\mathbb{S}^{m_2}(1/{\sqrt 2})$ in $\mathbb{R}^{m+2}$,
$m_1<m_2$, $m_1+m_2=m$. Then
\begin{equation}\label{eq: spectr_decomp_Cliff}
\varphi=\varphi_p+\varphi_q,
\end{equation}
where $\varphi_p, \varphi_q:\mathbb{S}^{m_1}(1/{\sqrt 2})
\times\mathbb{S}^{m_2}(1/{\sqrt 2})\to\mathbb{R}^{m+2}$,
$\varphi_p(x,y)=(x,0)$, $\varphi_q(x,y)=(0,y)$ and $
\Delta\varphi_p=2m_1\varphi_p$, $ \Delta\varphi_q=2m_2\varphi_q$.

Thus $\mathbb{S}^{m_1}(1/{\sqrt 2}) \times\mathbb{S}^{m_2}(1/{\sqrt
2})$ is a $2$-type submanifold of $\mathbb{R}^{m+2}$ with
eigenvalues $\lambda_p=2m_1$ and $\lambda_q=2m_2$, and it is
mass-symmetric, i.e. it has center of mass in the origin.

Since the eigenvalues of the torus $\mathbb{S}^{m_1}(1/{\sqrt 2})
\times\mathbb{S}^{m_2}(1/{\sqrt 2})$ are obtained as the sum of
eigenvalues of the spheres $\mathbb{S}^{m_1}(1/{\sqrt 2})$ and
$\mathbb{S}^{m_2}(1/{\sqrt 2})$, we conclude that $p=1$. Also,
$q=2$, i.e. $\mathbb{S}^{m_1}(1/{\sqrt 2})
\times\mathbb{S}^{m_2}(1/{\sqrt 2})$ has order $[1,2]$ in
$\mathbb{R}^{m+2}$, if and only if $m_2\leq 2(m_1+1)$. Note that
this holds, for example, for $\mathbb{S}^{1}(1/{\sqrt 2})
\times\mathbb{S}^{2}(1/{\sqrt 2})\subset\mathbb{S}^4$.

Moreover, it can be easily proved that all proper biharmonic
submanifolds in $\mathbb{S}^{m+1}$ obtained by means of the product
composition property, as submanifolds in $\mathbb{R}^{m+2}$, have
the spectral decomposition given by \eqref{eq: spectr_decomp_Cliff}.
\end{remark}

\subsubsection*{Other examples of proper biharmonic immersed submanifolds in
spheres}\quad

\noindent In \cite{S1} and \cite{AEMS} the authors studied the
proper biharmonic Legendre immersed surfaces and the proper
biharmonic $3$-di\-men\-sional anti-invariant immersed submanifolds
in Sasakian space forms. They obtained the explicit representations
of such submanifolds in the unit Euclidean $5$-dimensional sphere
$\mathbb{S}^5$.

\begin{theorem}[\cite{S1}]\label{th:Chen_T^2 in S^5}
Let $\phi:M^2\to\mathbb{S}^5$ be a proper biharmonic Legendre
immersion. Then the position vector field
$\varphi=\mathbf{i}\circ\phi=\varphi(u,v)$ of $M$ in $\mathbb{R}^6$
is given by
$$
\varphi(u,v)=\frac{1}{\sqrt 2}(e^{iu},ie^{-iu}\sin\sqrt 2 v,
ie^{-iu}\cos\sqrt 2 v),
$$
where $\mathbf{i}:\mathbb{S}^5\to\mathbb{R}^6$ is the canonical
inclusion.
\end{theorem}

\begin{remark}
The map $\phi$ is a full proper biharmonic Legendre embedding of a
$2$-dimensional flat torus $\mathbb{R}^2/\Lambda$ into
$\mathbb{S}^5$, where the lattice $\Lambda$ is generated by
$(2\pi,0)$ and $(0,\sqrt{2}\pi)$. It has constant mean curvature
$|H|=1/2$, it is not pseudo-umbilical and its mean curvature vector
field is not parallel. Moreover, $\varphi=\varphi_p+\varphi_{q}$,
where
$$
\varphi_{p}(u,v)=\frac{1}{\sqrt 2}(e^{iu},0,0)
$$
$$
\varphi_{q}(u,v)=\frac{1}{\sqrt 2}(0,ie^{-iu}\sin\sqrt 2 v,
ie^{-iu}\cos\sqrt 2 v)
$$
and $ \Delta \varphi_{p}=\varphi_{p}$, $ \Delta
\varphi_{q}=3\varphi_{q}$. Thus $\varphi$ is a $2$-type immersion in
$\mathbb{R}^6$ with eigenvalues $1$ and $3$. In this case, $p=1$ and
$q=3$, i.e. $\varphi$ is a $[1,3]$-order immersion in
$\mathbb{R}^6$.
\end{remark}

\begin{theorem}[\cite{AEMS}]
Let $\phi:M^3\to\mathbb{S}^5$ be a proper biharmonic anti-invariant
immersion. Then the position vector field
$\varphi=\mathbf{i}\circ\phi=\varphi(u,v,w)$ of $M$ in
$\mathbb{R}^6$ is given by
$$
\varphi(u,v,w)=\frac{1}{\sqrt 2} e^{iw}(e^{iu},ie^{-iu}\sin\sqrt 2
v, ie^{-iu}\cos\sqrt 2 v).
$$
\end{theorem}

\begin{remark}
The map $\phi$ is a full proper biharmonic anti-invariant immersion
from a $3$-dimen\-sional flat torus $\mathbb{R}^3/\Lambda$ into
$\mathbb{S}^5$, where the lattice $\Lambda$ is generated by
$(2\pi,0,0)$, $(0,\sqrt{2}\pi,0)$ and $(0,0,2\pi)$. It has constant
mean curvature $|H|=1/3$, is not pseudo-umbilical, but its mean
curvature vector field is parallel. Moreover,
$\varphi=\varphi_{p}+\varphi_{q}$, where
$$
\varphi_{p}(u,v,w)=\frac{1}{\sqrt 2} e^{iw}(e^{iu},0,0)
$$
$$
\varphi_{q}(u,v,w)=\frac{1}{\sqrt 2} e^{iw}(0,ie^{-iu}\sin\sqrt 2 v,
ie^{-iu}\cos\sqrt 2 v)
$$
and $\Delta \varphi_{p}=2\varphi_{p}$, $\Delta
\varphi_{q}=4\varphi_{q}$. Thus $\varphi$ is a $2$-type submanifold
of $\mathbb{R}^6$ with eigenvalues $2$ and $4$. It is easy to verify
that $\varphi$ is a $[2,4]$-order immersion in $\mathbb{R}^6$.

Since the immersion $\phi$ has parallel mean curvature vector field,
one could ask weather its image arises by means of the product
composition property. Indeed, it can be proved that, up to an
orthogonal transformation of $\mathbb{R}^6$ which commutes with the
usual complex structure, $\phi$ covers twice the proper biharmonic
submanifold $\mathbb{S}^1(1/\sqrt 2)\times \mathbb{S}^1(1/2)\times
\mathbb{S}^1(1/2)\subset\mathbb{S}^5$.
\end{remark}

\subsection{Classification results}

Some of the techniques used in order to obtain non-existence results
in the case of non-positively curved space forms were adapted to the
study of proper biharmonic submanifolds in spheres. Thus, in order
to approach the classification problem for proper biharmonic
hypersurfaces in spheres, the study was divided according to the
number of principal curvatures. For submanifolds of higher
codimension, supplementary conditions on the mean curvature vector
field were imposed. This leaded to a series of rigidity results,
which we enumerate below.

\subsubsection{Proper biharmonic hypersurfaces}

First, if $M$ is a proper biharmonic umbilical hypersurface in
$\mathbb{S}^{m+1}$, i.e. all its principal curvatures are equal,
then it is an open part of $\mathbb{S}^m(1/\sqrt{2})$.

Afterwards, proper biharmonic hypersurfaces with at most two
distinct principal curvatures were considered.
\begin{theorem}[\cite{BMO1}]\label{th: curb_med_const_2_curv_princ}
Let $M$ be a hypersurface with at most two distinct principal
curvatures in $\mathbb{S}^{m+1}$. If $M$ is proper biharmonic in
$\mathbb{S}^{m+1}$, then it has constant mean curvature.
\end{theorem}
By using this result, the classification of such hypersurfaces was
obtained.
\begin{theorem}[\cite{BMO1}]\label{th: classif_hypersurf_2_curv_princ}
Let $M^m$ be a hypersurface with at most two distinct principal
curvatures in $\mathbb{S}^{m+1}$. Then $M$ is proper biharmonic if
and only if it is an open part of $\mathbb{S}^{m}(1/\sqrt{2})$ or of
$\mathbb{S}^{m_1}(1/\sqrt{2})\times \mathbb{S}^{m_2}(1/\sqrt{2})$,
$m_1+m_2=m$, $m_1\neq m_2$.
\end{theorem}

Then followed the case of biharmonic hypersurfaces with at most
three distinct principal curvatures. In order to solve this problem,
the following property of proper biharmonic hypersurfaces in spheres
was needed.

\begin{proposition}[\cite{BMO1}]\label{th: curb_scal_hyp}
Let $M$ be a proper biharmonic hypersurface with constant mean
curvature $|H|$ in $\mathbb{S}^{m+1}$, $m\geq 2$. Then $M$ has
positive constant scalar curvature $ s=m^2(1+|H|^2)-2m$.
\end{proposition}

First a non-existence result was obtained.
\begin{theorem}[\cite{BMO3}]\label{th: non-exist_isoparam_3princ}
There exist no compact proper biharmonic hypersurfaces of constant
mean curvature and with three distinct principal curvatures
everywhere in the unit Euclidean sphere.
\end{theorem}
The proof relies on the fact that such hypersurfaces are
isoparametric, i.e. have constant principal curvatures with constant
multiplicities, and then, on the explicit expressions of the
principal curvatures.

We note that, in \cite{IIU}, the authors classified the
isoparametric proper biharmonic hypersurfaces in spheres.
\begin{theorem}[\cite{IIU}]
Let $M^m$ be an isoparametric hypersurface in $\mathbb{S}^{m+1}$.
Then $M$ is proper biharmonic if and only if it is an open part of
$\mathbb{S}^{m}(1/\sqrt{2})$ or of
$\mathbb{S}^{m_1}(1/\sqrt{2})\times \mathbb{S}^{m_2}(1/\sqrt{2})$,
$m_1+m_2=m$, $m_1\neq m_2$.
\end{theorem}

Compact proper biharmonic hypersurfaces in $\mathbb{S}^4$ were fully
classified.
\begin{theorem}[\cite{BMO3}]\label{th: hyper_S4}
The only compact proper biharmonic hypersurfaces in $\mathbb{S}^4$
are the hypersphere $\mathbb{S}^3(1/\sqrt{2})$ and the torus
$\mathbb{S}^1(1/\sqrt{2})\times\mathbb{S}^2(1/\sqrt{2})$.
\end{theorem}
The proof uses the fact that a proper biharmonic hypersurface in
$\mathbb{S}^4$ has constant mean curvature, and thus constant scalar
curvature, and a result in \cite{CHA1}.

\subsubsection{Proper biharmonic submanifolds of
codimension higher than one}

In higher codimension, it was proved that the proper biharmonic
pseudo-umbilical submanifolds, of dimension different from four, in
spheres have constant mean curvature. This result leaded to the
classification of proper biharmonic pseudo-umbilical submanifolds of
codimension two.
\begin{theorem}[\cite{BMO1}]\label{th: classif_pseudo_umb_codim2}
Let $M^m$ be a pseudo-umbilical submanifold in $\mathbb{S}^{m+2}$,
$m\neq 4$. Then $M$ is proper biharmonic in $\mathbb{S}^{m+2}$ if
and only if it is minimal in $\mathbb{S}^{m+1}(1/\sqrt{2})$.
\end{theorem}

Surfaces with parallel mean curvature vector field in $\mathbb{S}^n$
were also investigated.
\begin{theorem}[\cite{BMO1}]\label{th: classif_surf_parallelH}
Let $M^2$ be a surface with parallel mean curvature vector field in
$\mathbb{S}^n$. Then $M$ is proper biharmonic in $\mathbb{S}^n$ if
and only if it is minimal in $\mathbb{S}^{n-1}(1/\sqrt{2})$.
\end{theorem}

The above two results allowed the classification of proper
biharmonic constant mean curvature surfaces in $\mathbb{S}^4$.

\begin{theorem}[\cite{BO}]
The only proper biharmonic constant mean curvature surfaces in
$\mathbb{S}^4$ are the minimal surfaces in
$\mathbb{S}^3(1/\sqrt{2})$.
\end{theorem}

\begin{proof}
The key of the proof is to show that $\nabla^\perp H=0$, in order to
be able to apply Theorem \ref{th: classif_surf_parallelH}.

We assume that $\nabla^\perp H\neq 0$ and consider $\{E_1,E_2\}$
tangent to $M$ and $\{E_3, E_4=\frac{H}{|H|}\}$ normal to $M$, such
that $\{E_1,E_2,E_3,E_4\}$ constitutes a local orthonormal frame
field on $\mathbb{S}^4$. Using the connection $1$-forms w.r.t.
$\{E_1,E_2,E_3,E_4\}$ and the tangent part of the biharmonic
equation \eqref{caract_bih_spheres}, we get $A_4=0$, where $A_4$ is
the shape operator in direction of $E_4$.
Then we identify two cases:\\
(i) If $A_3=|H|\Id$, then $M$ is pseudo-umbilical and, by Theorem
\ref{th: classif_pseudo_umb_codim2}, it is minimal in
$\mathbb{S}^3(1/\sqrt{2})$. This implies that $\nabla^\perp
H=0$, and we have a contradiction.\\
(ii) If $A_3\neq|H|\Id$, then the Gauss and Codazzi equations lead
us to a contradiction and we conclude.
\end{proof}

\section{Properties of proper biharmonic submanifolds in spheres}\label{sec:3}

We begin this section by presenting some general properties of
proper biharmonic submanifolds with parallel mean curvature vector
field in spheres, which are consequences of
\eqref{caract_bih_spheres} and of the Codazzi and Gauss equations,
respectively.

\begin{proposition}\label{prop: prop||MC}
Let $M$ be a proper biharmonic submanifold with parallel mean
curvature vector field in $\mathbb{S}^n$. Then
\begin{itemize}
\item[(i)] $|A_H|^2=m|H|^2$, and it is constant,
\item[(ii)] $\trace \nabla A_H=0$,
\item[(iii)] $\langle\trace (\nabla^\perp B)(X,\cdot, A_H(\cdot)),H\rangle=\langle\trace (\nabla^\perp B)(\cdot,X, A_H(\cdot)),H\rangle=0$, for
all $X\in C(TM)$.
\end{itemize}
\end{proposition}

\begin{proposition}\label{prop: prop||MC_from_Gauss_eq}
Let $M$ be a proper biharmonic submanifold with parallel mean
curvature vector field in $\mathbb{S}^n$. Let $p$ be an arbitrary
point on $M$ and consider $\{e_i\}_{i=1}^m$ to be an orthonormal
basis of eigenvectors for $A_H$ in $T_pM$ . Denote by
$\{a_i\}_{i=1}^m$ the eigenvalues of $A_H$ at $p$. Then, at $p$,
\begin{itemize}
\item[(i)] $m|H|^2=\displaystyle{\sum_{i=1}^m a_i=\sum_{i=1}^m (a_i)^2}$,
\item[(ii)] $\displaystyle{(2m-1)m|H|^2=\frac{1}{2}\sum_{i,j=1}^m(a_i+a_j)(K_{ij}+|B(e_i,e_j)|^2)}$,
\item[(iii)] $\displaystyle{(m-1+m|H|^2)m|H|^2=\sum_{i,j=1}^m a_i
a_j(K_{ij}+|B(e_i,e_j)|^2)}$,
\end{itemize}
where $K_{ij}$ denotes the sectional curvature of the $2$-plane
tangent to $M$ generated by $e_i$ and $e_j$.
\end{proposition}

For what concerns proper biharmonic constant mean curvature
submanifolds in spheres, a partial classification result was
obtained.

\begin{theorem}[\cite{O3}]\label{th: classif_bih const mean}
Let $M$ be a proper biharmonic submanifold with constant mean
curvature in $\mathbb{S}^n$. Then $|H|\in(0,1]$. Moreover, if
$|H|=1$, then $M$ is a minimal submanifold of a hypersphere
$\mathbb{S}^{n-1}(1/\sqrt{2})\subset\mathbb{S}^n$.
\end{theorem}

Also, the properties regarding the type of the main examples
previously presented are not casual. In fact, Theorem \ref{th:
classif_bih const mean} was extended by establishing a general link
between compact proper biharmonic constant mean curvature
submanifolds in spheres and finite type submanifolds in the
Euclidean space.

\begin{theorem}[\cite{BMO1}]\label{th: tip_subv} Let $M^m$ be a compact constant
mean curvature, $|H|\in(0,1]$, sub\-manifold in  $\mathbb{S}^n$.
Then $M$ is proper biharmonic if and only if

either
\begin{itemize}
\item[(i)]  $|H|=1$ and $M$ is a 1-type
submanifold of $\mathbb{R}^{n+1}$ with eigenvalue $\lambda=2m$ and
center of mass of norm equal to $1/\sqrt{2}$,
\end{itemize}

or
\begin{itemize}
\item[(ii)] $|H|\in (0,1)$ and $M$ is a mass-symmetric $2$-type
submanifold of $\mathbb{R}^{n+1}$ with eigenvalues
$\lambda_{p}=m(1-|H|)$ and $\lambda_{q}=m(1+|H|)$.
\end{itemize}
\end{theorem}

This can be further used in order to obtain some necessary
conditions that compact proper biharmonic submanifolds with constant
mean curvature in spheres must fulfill.

\begin{corollary}\label{cor: conseq_type}
Let $M^m$ be a compact proper biharmonic constant mean curvature,
$|H|\in(0,1)$, sub\-manifold in $\mathbb{S}^n$. Then
\begin{itemize}
\item[(i)] $\lambda_1\leq m(1-|H|)$, where $\lambda_1$ is the first non-zero eigenvalue
of the Laplacian on $M$,
\item[(ii)] if $\rm{Ricci}(X,X)\geq cg(X,X)$, for all $X\in C(TM)$,
where $c>0$, we have $c\leq (m-1)(1-|H|)$.
\end{itemize}
\end{corollary}

\begin{proof}
(i) From Theorem \ref{th: tip_subv} it follows that the inclusion
map of $M$ in $\mathbb{R}^{n+1}$, $\varphi:M\to\mathbb{R}^{n+1}$,
decomposes as $\varphi=\varphi_p+\varphi_q$, where $\Delta
\varphi_p=\lambda_p\varphi_p$, $\Delta
\varphi_q=\lambda_q\varphi_q$, $\lambda_p=m(1-|H|)$ and
$\lambda_q=m(1+|H|)$. Conclusively, $m(1-|H|)$ is a non-zero
eigenvalue of the Laplacian on $M$, and thus $\lambda_1\leq
m(1-|H|)$.

(ii) The condition $\rm{Ricci}(X,X)\geq cg(X,X)$, for all $X\in
C(TM)$, implies, by a well-known result of Lichnerowicz (see
\cite{BGM}), that $\lambda_1\geq \frac{m}{m-1}c$. This, together
with (i), leads to the conclusion.
\end{proof}

We shall need the following result in order to obtain a refinement
of Theorem \ref{th: classif_bih const mean}.

\begin{theorem}[\cite{HZL}]\label{th: Haizhong Li}
Let $M$ be a compact hypersurface with constant normalized scalar
curvature $r=\frac{s}{m(m-1)}$ in $\mathbb{S}^{m+1}$. If
\begin{itemize}
\item[(i)] $r\geq 1$,
\item[(ii)] the squared norm $|B|^2$ of the second fundamental
form of $M$ satisfies
\begin{equation}\label{eq: eval_Li}
m(r-1)\leq |B|^2\leq(m-1)\frac{m(r-1)+2}{m-2}+\frac{m-2}{m(r-1)+2},
\end{equation}
\end{itemize}
then either $ |B|^2=m(r-1)$ and $M$ is a totally umbilical
hypersurface; or
$$
|B|^2=(m-1)\frac{m(r-1)+2}{m-2}+\frac{m-2}{m(r-1)+2}
$$
and $M=\mathbb{S}^1(\sqrt{1-c^2})\times\mathbb{S}^{m-1}(c)$, with
$c^2=\frac{m-2}{mr}$.
\end{theorem}

We get the following theorem.

\begin{theorem}{\label{th: raf_tip_hyper}}
Let $M^m$, $m\geq 4$, be a compact proper biharmonic constant mean
curvature hypersurface in $\mathbb{S}^{m+1}$. Then
$|H|\in(0,\frac{m-2}{m}]\cup\{1\}$. Moreover,
\begin{itemize}
\item[(i)]  $|H|=1$ if and only if $M=\mathbb{S}^m(1/\sqrt{2})$,
\end{itemize}
and
\begin{itemize}
\item[(ii)] $|H|=\frac{m-2}{m}$ if and only if
$M=\mathbb{S}^1(1/\sqrt{2})\times\mathbb{S}^{m-1}(1/\sqrt{2})$.
\end{itemize}
\end{theorem}

\begin{proof}
Since $M$ is proper biharmonic with constant mean curvature $|H|$,
Theorem \ref{th: bih subm S^n} implies that
\begin{equation}\label{eq: normB^2}
|B|^2=|A|^2=m.
\end{equation}
We shall denote, for convenience, $t=m|H|^2-1$.

Suppose that $|H|\in(\frac{m-2}{m},1)$, which is equivalent to
$t\in\left(\frac{(m-4)(m-1)}{m},m-1\right)$. By using Proposition
\ref{th: curb_scal_hyp}, we obtain that
\begin{equation}\label{eq: r}
r=1+\frac{t}{m-1}.
\end{equation}
Condition (i) of Theorem \ref{th: Haizhong Li} is equivalent to
$t\geq 0$, which is satisfied. Also, using \eqref{eq: normB^2},
since $t< m-1$, the first inequality of \eqref{eq: eval_Li} is
satisfied. The second inequality of \eqref{eq: eval_Li} becomes
$$
0\leq mt^2-(m^2-6m+4)t-(m-4)(m-1)
$$
and it is satisfied since $t>\frac{(m-4)(m-1)}{m}$. We are now in
the hypotheses of Theorem \ref{th: Haizhong Li} and we get $r=2$,
i.e. $|H|=1$, or $r=\frac{2(m-2)}{m}$, i.e. $|H|=\frac{m-2}{m}$,
thus we have a contradiction. Conclusively, we obtain
$|H|\in(0,\frac{m-2}{m}]\cup\{1\}$.

Case (i) is given by Theorem \ref{th: classif_bih const mean}. It
can also be proved by using Theorem \ref{th: Haizhong Li}.

For (ii), as we have already seen, if
$M=\mathbb{S}^1(1/\sqrt{2})\times\mathbb{S}^{m-1}(1/\sqrt{2})$, then
$|H|=\frac{m-2}{m}$. Conversely, if $|H|=\frac{m-2}{m}$, then
$r=\frac{2(m-2)}{m}$, and we are in the hypotheses of Theorem
\ref{th: Haizhong Li}, thus we conclude.
\end{proof}

\section{Open problems}

In view of all the above results the following conjectures were
proposed.
\begin{conjecture}[\cite{BMO1}]
The only proper biharmonic hypersurfaces in $\mathbb{S}^{m+1}$ are
the open parts of hyperspheres $\mathbb{S}^{m}(1/\sqrt{2})$ or of
generalized Clifford tori $\mathbb{S}^{m_1}(1/\sqrt{2})\times
\mathbb{S}^{m_2}(1/\sqrt{2})$, $m_1+m_2=m$, $m_1\neq m_2$.
\end{conjecture}
\begin{conjecture}[\cite{BMO1}]
Any proper biharmonic submanifold in $\mathbb{S}^n$ has constant
mean curvature.
\end{conjecture}

\section{Further remarks}

There is an interesting link between the proper biharmonic
hypersurfaces in $\mathbb{S}^{m+1}$ and the $II$-minimal
hypersurfaces. We briefly recall here the notion of $II$-minimal
hypersurfaces (see \cite{SHSV}). We denote by $\mathcal{E}$ the set
of all hypersurfaces in a semi-Riemannian manifold $(N,h)$ for which
the first, as well as the second, fundamental form is a
semi-Riemannian metric. The critical points of the area functional
of the second fundamental form
$$
\mathrm{Area}_{II}:\mathcal{E}\to\mathbb{R}, \qquad
\mathrm{Area}_{II}(M)=\int_M \sqrt{|\mathrm{det} A|}\,v_g
$$
are called $II$-minimal.  According to \cite{SHSV}, we have
\begin{proposition}
Let $\mathbb{S}^m(a)$ be the hypersphere of radius $a\in(0,1)$ in
$\mathbb{S}^{m+1}$. The following are equivalent
\begin{itemize}
\item[(i)] $\mathbb{S}^m(a)$ is proper biharmonic,
\item[(ii)] $\mathbb{S}^m(a)$ is $II$-minimal,
\item[(iii)] $a=1/\sqrt{2}$.
\end{itemize}
\end{proposition}

\begin{proposition}
Let $M=\mathbb{S}^{m_1}(a_1)\times \mathbb{S}^{m_2}(a_2)$,
$a_1\in(0,1)$, $a_1^2+a_2^2=1$, be the generalized Clifford torus in
$\mathbb{S}^{m+1}$, $m_1+m_2=m$. The following are equivalent
\begin{itemize}
\item[(i)] $M$ is proper biharmonic,
\item[(ii)] $M$ is $II$-minimal and non-minimal,
\item[(iii)] $a_1=a_2=1/\sqrt{2}$ and $m_1\neq m_2$.
\end{itemize}
\end{proposition}

\section*{Acknowledgements}

The authors would like to thank Professors I. Dimitric and S.
Verpoort for helpful suggestions and discussions.

%\printindex
\end{document}